\documentclass[12pt]{article}
\usepackage{amssymb}
\usepackage{amsmath,amsthm}
\usepackage[latin1]{inputenc}
\usepackage[dvips]{graphicx}

 \newtheorem{remark}{Remark}

 \newtheorem{theorem}[remark]{Theorem}
 
 \newtheorem{corollary}[remark]{Corollary}

\title{On the  partition dimension of trees}

\author{ Juan A.
Rodr\'{\i}guez-Vel\'{a}zquez$^{1}$, Ismael G. Yero$^{1}$ and\\
Magdalena Lema\'{n}ska$^{2}$\\
\\
$^1${\small Departament d'Enginyeria Inform\`{a}tica i Matem\`{a}tiques }\\
{\small Universitat Rovira i Virgili,  Av. Pa\"{\i}sos Catalans 26,
43007 Tarragona, Spain.} \\{\small ismael.gonzalez\@@urv.cat,
juanalberto.rodriguez\@@urv.cat} \\
$^2${\small Department of Technical Physics and Applied Mathematics}
\\ {\small Gdansk University of Technology, ul. Narutowicza 11/12
80-233 Gdansk, Poland }\\ {\small magda\@@mifgate.mif.pg.gda.pl}}

\begin{document}

\maketitle

\begin{abstract}
Given an ordered partition $\Pi =\{P_1,P_2, ...,P_t\}$ of the
vertex set $V$ of a connected graph $G=(V,E)$, the  \emph{partition representation} of a vertex
$v\in V$ with respect to the partition $\Pi$ is the vector
$r(v|\Pi)=(d(v,P_1),d(v,P_2),...,d(v,P_t))$, where $d(v,P_i)$ represents the distance between the vertex $v$ and
the set $P_i$. A partition  $\Pi$ of $V$ is a  \emph{resolving partition} of $G$ if different vertices of $G$ have different partition representations, i.e.,  for every pair
of vertices $u,v\in V$, $r(u|\Pi)\ne r(v|\Pi)$. The  \emph{partition
dimension} of $G$ is the minimum number of sets in any resolving
partition of $G$. In this paper we obtain several tight bounds
on the partition dimension of trees.\end{abstract}



\section{Introduction}

 The concepts of resolvability and location in graphs were described
independently by Harary and Melter \cite{harary} and Slater
\cite{leaves-trees}, to define the same structure in a
graph. After these papers were published several authors
developed diverse theoretical works about this topic
\cite{pelayo1,pelayo2,chappell,chartrand,chartrand1,chartrand2,fehr,harary,haynes,landmarks}.
 Slater described the usefulness of these ideas into long range
aids to navigation \cite{leaves-trees}.
Also, these concepts  have some applications in chemistry for representing chemical compounds \cite{pharmacy1,pharmacy2} or to problems of
pattern recognition and image processing, some of which involve the use of hierarchical data structures \cite{Tomescu1}.
Other applications of this concept to
navigation of robots in networks and other areas appear in
\cite{chartrand,robots,landmarks}. Some variations on resolvability
or location have been appearing in the literature, like those about
conditional resolvability \cite{survey}, locating domination
\cite{haynes}, resolving domination \cite{brigham} and resolving
partitions \cite{chappell,chartrand2,fehr}.

Given a graph $G=(V,E)$ and a set of vertices
$S=\{v_1,v_2,...,v_k\}$ of $G$, the  \emph{metric representation} of a
vertex $v\in V$ with respect to $S$ is the vector
$r(v|S)=(d(v,v_1),d(v,v_2),...,d(v,v_k))$, where $d(v,v_i)$ denotes the distance between the vertices $v$ and
$v_i$, $1\le i\le k$. We say that $S$ is a  \emph{resolving set} of $G$ if different vertices of $G$ have different metric representations, i.e.,  for every
pair of vertices $u,v\in V$, $r(u|S)\ne r(v|S)$. The  \emph{metric
dimension}\footnote{Also called locating number.} of $G$ is the
minimum cardinality of any resolving set of $G$, and it is
denoted by $dim(G)$. The metric dimension of graphs is studied in
\cite{pelayo1,pelayo2,chappell,chartrand,chartrand1,tomescu}.

Given an ordered partition $\Pi =\{P_1,P_2, ...,P_t\}$ of the
vertices of $G$, the  \emph{partition representation} of a vertex
$v\in V$ with respect to the partition $\Pi$ is the vector
$r(v|\Pi)=(d(v,P_1),d(v,P_2),...,d(v,P_t))$, where $d(v,P_i)$, with
$1\leq i\leq t$, represents the distance between the vertex $v$ and
the set $P_i$, i.e.,  $d(v,P_i)=\min_{u\in P_i}\{d(v,u)\}$. We say
that $\Pi$ is a \emph{resolving partition} of $G$ if different vertices of $G$ have different partition representations, i.e.,  for every pair
of vertices $u,v\in V$, $r(u|\Pi)\ne r(v|\Pi)$. The  \emph{partition
dimension} of $G$ is the minimum number of sets in any resolving
partition of $G$ and it is denoted by $pd(G)$. The partition
dimension of graphs is studied in \cite{chappell,chartrand2,fehr,tomescu}.

\section{The partition dimension of trees}

It is natural to think that the partition dimension and metric dimension are related; in \cite{chartrand2} it was shown that for any
nontrivial connected graph $G$ we have
\begin{equation}\label{partition-dimension}pd(G)\le  dim(G) + 1.\end{equation}
We know that the partition dimension of any path is two. That is, for any path graph $P,$ it follows $pd(P)=dim(P)+1=2$. A formula for the dimension of trees that are not  paths has been established in \cite{chartrand,harary,leaves-trees}. In order to present this formula, we need additional definitions. A vertex of degree at least $3$ in a tree $T$ will be called a \emph{major vertex} of $T$.
Any leaf $u$ of $T$ is said to be a \emph{terminal vertex} of a major vertex $v$ of $T$ if
$d(u, v)<d(u,w)$ for every other major vertex $w$ of $T$. The \emph{terminal degree}  of
a major vertex $v$ is the number of terminal vertices of $v$. A major vertex $v$ of $T$ is an
\emph{exterior major vertex} of $T$ if it has positive terminal degree.
\begin{figure}[ht]
  \centering
  \includegraphics[width=0.3\textwidth]{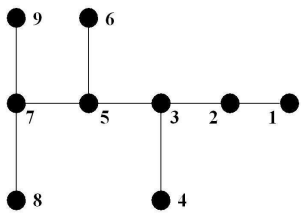}
  \caption{In this tree the vertex 3 is an exterior major vertex of terminal degree two: 1 and 4 are terminal vertices of 3. }\label{segundomejor}
\end{figure}

Let $n_1(T)$ denote the number of leaves of $T$, and let $ex(T)$ denote the number
of exterior major vertices of $T$. We can now state the formula for the dimension of a tree \cite{chartrand,harary,leaves-trees}:  if $T$ is a tree that is not a path, then
\begin{equation}\label{bounDimUsandoExterior}dim(T) = n_1(T) - ex(T).\end{equation}
As a consequence, if $T$ is a tree that is not a path, then
\begin{equation}\label{bounPartDimUsandoExterior}pd(T) \le  n_1(T) - ex(T)+1.\end{equation}


The above bound is tight, it is achieved for the graph in Figure \ref{segundomejor} where $\Pi=\{\{8\},\{4,9\},\{1,2,3,5,6,7\}\}$ is
a resolving partition and  $pd(T)=3$.
However, there are graphs for which the following bound gives better result than bound (\ref{bounPartDimUsandoExterior}), for instance, the graph in Figure \ref{2arbolxx}.

\begin{figure}[ht]
\centering
    \includegraphics[width=0.33\textwidth]{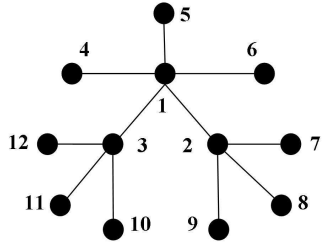}
  \caption{
$\Pi=\{\{1,4,9,12\},\{3,5,8,11\},\{2,6,7,10\}\}$ is
a resolving partition.
}\label{2arbolxx}
\end{figure}

Let $S=\{s_1,s_2,...,s_{\kappa}\}$ be the set of exterior major
vertices of $T=(V,E)$ with terminal degree greater than one, let $\{s_{i1},s_{i2},..., s_{il_i}\}$ be the set
of terminal vertices of $s_i$ and let $\tau=\max_{1\le i\le \kappa}\{l_i\}$. With the above notation we have the following result.

\begin{theorem}\label{ThPdTrees}
For any tree $T$  which is not a path, $$pd(T)\le
\kappa+\tau-1.$$
\end{theorem}

\begin{proof}
 For a terminal vertex
$s_{ij}$ of a major vertex $s_i\in S$ we denote by $S_{ij}$ the set of vertices of $T$, different from $s_i$, belonging to the $s_i-s_{ij}$ path. If $l_i< \tau-1$, we assume $S_{ij}=\emptyset$ for every $j\in \{l_i+1,...,\tau-1 \}$.
Now for every $j\in \{2,...,\tau -1\}$, let $B_j=\cup_{i=1}^{\kappa}S_{ij}$  and,  for every $i\in \{1,...,\kappa\}$, let
$A_i=S_{i1}$.
Let us show that
$\Pi=\{A,A_1,A_2,...A_{\kappa},B_2,...,B_{\tau-1}\}$ is a
resolving partition of $T$, where $A=V-\left(\displaystyle\left(\cup_{i=1}^{\kappa}A_i\right)\cup \left(\cup_{j=2}^{\tau-1}B_j\right)\right)$.  We consider two different vertices
$x,y\in V$. Note that if $x$ and $y$ belong to different sets of
$\Pi$, we have $r(x|\Pi)\ne r(y|\Pi)$.

Case 1: $x,y\in S_{ij}$. In this case, $d(x,A)=d(x,s_i)\ne d(y,s_i)=d(y,A)$.

Case 2: $x\in S_{ij}$ and $y\in S_{kl}$, $i\ne k$. If $j=1$ or $l=1$, then $x$ and $y$ belong to different sets of $\Pi$. So we suppose $j\ne1$ and $l\ne1$. Hence, if  $d(x,A_i)=d(y,A_i)$, then
\begin{align*}
d(x,A_k)&=d(x,s_i)+d(s_i,s_k)+1\\
&=d(x,A_i)+d(s_i,s_k)\\
&=d(y,A_i)+d(s_i,s_k)\\
&=d(y,s_k)+2d(s_k,s_i)+1\\
&=d(y,A_k)+2d(s_k,s_i)\\
&>d(y,A_k).
\end{align*}

Case 3:  $x\in S_{i\tau}$ and $y\in A- \cup_{l=1}^{\kappa} S_{l\tau}$. If  $d(x,A_i)=d(y,A_i)$, then $d(x,s_i)=d(y,s_i)$. Since $y\notin S_{l\tau}$, $l\in \{1,...,\kappa\}$, there exists $A_j\in \Pi$,  $j\ne i$, such that $s_i$ does not belong to the $y-s_j$ path. Now let $Y$ be the set of vertices belonging to the $y-s_j$ path, and let $v\in Y$ such that $d(s_i,v)=\displaystyle\min_{u\in Y}\{d(s_i,u)\}$. Hence,
\begin{align*}
d(x,A_j)&=d(x,s_i)+d(s_i,v)+d(v,s_j)+1\\
&=d(y,s_i)+d(s_i,v)+d(v,s_j)+1\\
&=d(y,v)+2d(v,s_i)+d(v,s_j)+1\\
&=d(y,A_j)+2d(v,s_i)\\
&>d(y,A_j).
\end{align*}

Case 4: $x,y\in A'=A- \cup_{l=1}^{\kappa} S_{l\tau}$. If for some exterior major vertex $s_i\in S$,
the vertex $x$ belongs to the $y-s_i$ path or the vertex $y$ belongs to the $x-s_i$ path, then $d(x,A_i)\ne d(y,A_i)$. Otherwise, there exist at least two exterior major vertices $s_i$, $s_j$ such that the $x-y$ path and the $s_i-s_j$ path share more than one vertex (if not, then $x,y\notin A'$). Let $W$ be the set of vertices belonging to the $s_i-s_j$ path. Let $u, v\in W$ such that $d(x,u)=\min_{z\in W}\{d(x,z)\}$ and $d(y,v)=\min_{z\in W}\{d(y,z)\}$. We suppose, without loss of generality, that $d(s_i,u)>d(v,s_i)$.  Hence, if $d(x,v)= d(y,v)$, then $d(x,u)\ne d(y,u)$, and if $d(x,u)= d(y,u)$, then $d(x,v)\ne d(y,v)$. We have,
\begin{align*}d(x,A_j)&
=d(x,u)+d(u,s_j)+1\\
&\ne d(y,u)+d(u,s_j)+1\\
&=d(y,A_j)
\end{align*}
or
\begin{align*}d(x,A_i)&=d(x,v)+d(v,s_i)+1\\
&\ne d(y,v)+d(v,s_i)+1\\
&=d(y,A_i).
\end{align*}

Therefore, for  different vertices $x,y\in V,$ we have
$r(x|\Pi)\ne r(y|\Pi)$.
\end{proof}

One example where $pd(T) = \kappa+\tau-1$ is the tree in Figure \ref{segundomejor}.


Any vertex adjacent to a leaf of a tree $T$ is called a support vertex. In the following result $\xi$ denotes the number of support vertices of $T$ and  $\theta$ denotes  the maximum number of leaves adjacent to a support vertex of $T$.

\begin{corollary}\label{ThPdTrees2}
For any tree $T$  of order $n\ge 2$,  $pd(T)\le \xi+\theta-1$.
\end{corollary}

\begin{proof}
If $T$ is a path, then $\xi=2$ and $\theta=1$, so the result follows. Now we suppose $T$ is not a path.
Let $v$ be an exterior major vertex of terminal degree $\tau$. Let $x$ be the number of leaves adjacent to $v$ and let $y=\tau-x$. Since $\kappa+y\le \xi$ and $x\le \theta$, we deduce $\kappa+\tau\le \xi+\theta$.
\end{proof}


The above bound is achieved, for instance, for the graph of order  six composed by two support vertices $a$ and $b$, where $a$ is adjacent to $b$, and four leaves; two of them are adjacent to $a$ and the other two leaves are adjacent to $b$. One example of graph for which Theorem \ref{ThPdTrees}  gives better result than Corollary \ref{ThPdTrees2} is the graph in Figure \ref{segundomejor}.

Since the number of leaves, $n_1(T),$ of a tree $T$ is bounded below by $\xi+\theta-1$,   Corollary \ref{ThPdTrees2} leads to the following bound.
\begin{remark}
For any tree $T$  of order $n\ge 2$, $pd(T)\le n_1(T)$.
\end{remark}

Now we are going to characterize all the trees for which $pd(T)= n_1(T)$.
It was shown in \cite{chartrand2} that $pd(G)=2$ if and only if the graph $G$ is a path. So by the above remark we obtain the following result.
\begin{remark}\label{RemarkPDtreesthreeleaves}
Let  $T$  be a tree of order $n\ge 4$. If $n_1(T)=3$, then $ pd(T)= 3$.
\end{remark}
\begin{theorem}
Let $T$ be a tree  with $n_1(T)\ge 4$. Then
$pd(T)=n_1(T)$ if and only if $T$ is the star graph.
\end{theorem}

\begin{proof}$\;$
If $T=S_n$ is a star graph,  it is clear that $pd(T)=n_1(T)$. Now, let
$T=(V,E)\ne S_n$, such that $pd(T)=n_1(T)\ge 4$. Note that by
(\ref{bounPartDimUsandoExterior}) we have $ex(T)=1$. Let  $t=n_1(T)$ and let $\Omega=\{u_1, u_2, ...,
u_t\}$ be the set of leaves of $T$. Let $u\in
V$ be the unique exterior major vertex of $T$. Let us suppose,
without loss of generality, $u_t$ is a leaf of $T$ such that
$d(u_t,u)=\max_{u_i\in \Omega}\{d(u_i,u)\}$.

For the leaves $u_1,u_2,u_t\in \Omega$ let the paths
$P=uu_{t1}u_{t2}...u_{tr_t}u_t$, $Q=uu_{11}u_{12}...u_{1r_1}u_1$ and
$R=uu_{21}u_{22}...u_{2r_2}u_2$. Now, let us form the partition
$\Pi=\{A_1,A_2,...,A_{t-2},A\}$, such that
$A_1=\{u_{11},u_{12},...,u_{1r_1},u_1,u_{t2},u_{t3},...,u_{tr_t},u_t\}$,
$A_2=\{u_{21},u_{22},...,u_{2r_2},u_2,u_{t1}\}$, $A_i=\{u_i\}$,
$i\in \{3,..,t-2\}$ and $A=V-\cup_{i=1}^{t-2}A_i$. Let us consider
two different vertices $x,y\in V$. Hence, we have the following
cases,

Case 1: $x,y\in A_1$. Let us suppose $x\in P$ and $y\in Q$. If
$d(x,A_2)=d(y,A_2)$, then we have
\begin{align*}
d(x,A)=&d(x,u_{t1})+1\\
=&d(x,A_2)+1\\
=&d(y,A_2)+1\\
=&d(y,A)+2\\
>&d(y,A).
\end{align*}
Now, if $x,y\in P$ or $x,y\in Q$, then $d(x,A)\ne d(y,A)$.

Case 2: $x,y\in A_2$. If $x=u_{t1}$ or $y=u_{t1}$, then let us
suppose for instance, $x=u_{t1}$, so we have $d(x,A_1)=1<2\le
d(y,A_1)$. On the contrary, if $x,y\in R$, then $d(x,A)\ne d(y,A)$.

Case 3: $x,y\in A$. If $d(x,A_1)=d(y,A_1)$, then $t\ge 5$ and there exists a leaf
$u_i$, $i\ne 1,2,t-1,t$, such that $d(x,A_i)=d(x,u_i)\ne
d(y,u_i)=d(y,A_i)$.

Therefore, for different vertices $x,y\in V$ we have $r(x|\Pi)\ne r(y|\Pi)$ and
$\Pi$ is a resolving partition in $T$, a contradiction.
\end{proof}

\begin{figure}[ht]
  \centering
\includegraphics[width=0.35\textwidth]{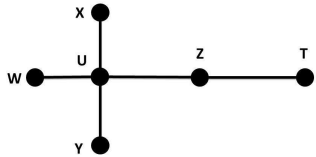}
\vspace{-1cm}
  \caption{A Comet graph where $3=\theta=Pd(T).$ }\label{polemico}
\end{figure}
Let $T$ be the Comet graph showed in Figure \ref{polemico}. A resolving partition for $T$  is $\Pi=\{A_1,A_2,A_3\}$,
where $A_1=\{x,t\}$, $A_2=\{y,z\}$ and $A_3=\{u,w\}$. In this case, $Pd(T)=3\theta.$

\begin{remark}
For any tree $T$  of order $n\ge 2$, $ pd(T)\ge \theta $.
\end{remark}

\begin{proof}
Since different leaves adjacent to the same support vertex must belong to different sets of a resolving partition, the result follows.
\end{proof}

Other examples  where $pd(T)= \theta$ are the star graphs and the graph in Figure \ref{2arbolxx}.

\begin{theorem}
Let $T$ be a  tree. If every vertex belonging to the path between two exterior major vertices of terminal degree greater than one is an exterior major vertex of terminal degree greater than one, then $$pd(T)\le \max\{\kappa,\tau+1\}.$$
\end{theorem}

\begin{proof}
If $T$ is a path, then $\tau= 2$ and $\kappa=1$, so the result follows.   We suppose $T=(V,E)$ is not a path. Let
$S=\{s_1,s_2,...,s_{\kappa}\}$ be the set of exterior major vertices
of $T$ with terminal degree greater than one and let $B_i=\{s_i\}$, $i=1,...,\kappa$. If $\kappa<\tau+1$, then for $i\in \{\kappa+1,...,\tau+1\}$ we assume $B_i=\emptyset$. Let $l_i$ be the terminal degree of $s_i$,  $i\in \{1,...,\kappa\}$.
If $l_i<i$, then we denote by $\{s_{i1},...,s_{il_i}\}$   the set of terminal
vertices of $s_i$. On the contrary, if $l_i\ge i$, then the set of terminal
vertices of $s_i$ is denoted by $\{s_{i1},...,s_{ii-1},s_{ii+1},...,s_{il_i+1}\}$.
 Also, for a terminal vertex $s_{ij}$ of a major
vertex $s_i$ we denote by $S_{ij}$ the set of vertices of $T$,
different from $s_i$, belonging to the $s_i-s_{ij}$ path.
Moreover, we assume $S_{ij}=\emptyset$ for the following three cases: (1) $i=j$, (2) $i\le l_i<\tau$ and
$j\in \{l_i+2,..., \tau+1\}$, and (3) $i> l_i$ and
$j\in \{l_i+1,..., \tau+1\}$.
Now, let $t=\max\{\kappa,\tau+1\}$ and let  $\Pi=\{A_1,A_2,...,A_{t}\}$ be composed by the sets
$A_i=B_i\cup \left( \cup_{j=1}^{\kappa} S_{ji}\right)$, $i=1,...,t$. Since every vertex belonging to the path between two exterior major vertices of terminal degree greater than one, is an exterior major vertex of terminal degree greater than one, then  $\Pi$ is a  partition of $V$.

Let us show that $\Pi$ is a resolving partition. Let $x,y\in V$ be different vertices of $T$. If
$x,y\in A_i$, we have the following three cases.\\
\noindent Case 1: $x,y\in S_{ji}$. In this case $d(x,A_j)=d(x,s_j)\ne d(y,s_j)=d(y,A_j)$.\\
\noindent Case 2: $x\in S_{ji}$ and $y\in S_{ki}$,  $j\ne k$. If $d(x,A_k)=d(y,A_k)$ we have $d(y,A_j)>d(y,s_k)
=d(y,A_k)=d(x,A_k)>d(x,s_j)=d(x,A_j).$\\
\noindent Case 3: $x=s_i$ and $y\in S_{ji}$. As $s_i$ has at least two
terminal vertices, there exists a terminal vertex $s_{il}$ of $s_i$,
$l\ne j$, such that $d(x,A_l)=d(x,S_{il})=1$. Hence,
$d(y,A_l)>d(y,s_j)\ge 1=d(x,A_l)$.
Therefore, for different vertices $x,y\in V$, we have $r(x|\Pi)\ne r(y|\Pi)$.
\end{proof}

The above bound is achieved, for instance, for the graph in Figure \ref{arbolcotamax}.

\begin{figure}[ht]
  \centering
  \includegraphics[width=0.3\textwidth]{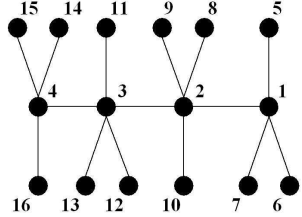}\\
  \caption{$\Pi=\{\{1,8,11,14\},\{2,5,12,15\},\{3,6,9,16\},\{4,7,10,13\}\}$ is a resolving partition.}\label{arbolcotamax}
\end{figure}

\section{On the partition dimension of generalized trees}

A cut vertex in a graph is a vertex whose removal increases the number of components of the graph and an extreme vertex is a vertex such that its closed neighborhood forms a complete graph. Also, a block is a maximal biconnected subgraph of the graph. Now, let $\mathfrak{F}$ be the family of sequences of connected graphs $G_1,G_2,...,G_k$, $k\ge 2$, such that $G_1$ is a complete graph $K_{n_1}$, $n_1\ge 2$, and $G_i$, $i\ge 2$, is obtained recursively from $G_{i-1}$ by adding a complete graph $K_{n_i}$, $n_i\ge 2$, and identifying a vertex of $G_{i-1}$ with a vertex in $K_{n_i}$.

From this point we will say that a connected graph $G$ is a \emph{generalized tree} if and only if there exists a sequence $\{G_1,G_2,...,G_k\}\in \mathfrak{F}$ such that $G_k=G$ for some $k\ge 2$. Notice that in these generalized trees every vertex is either, a cut vertex or an extreme vertex. Also, every complete graph used to obtain the generalized tree is a block of the graph. Note that if every $G_i$ is isomorphic to $K_2$, then $G_k$ is a tree, thus justifying the terminology used. In this section we will be centered in the study of partition dimension of generalized trees.

\begin{figure}[ht]
  \centering
  \includegraphics[width=0.3\textwidth]{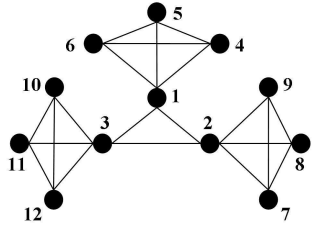}
  \caption{$\Pi=\{\{4\},\{7\},\{10\},\{5,8,11\},\{1,2,3,6,9,12\}\}$ is a resolving partition for the generalized tree.}\label{cotagentree}
\end{figure}

Let $G=(V,E)$ be a generalized tree and let $R_1,R_2,...,R_k$ be the blocks of $G$. A cut vertex $v\in V$ is a \emph{support cut vertex} if there is at least a block $R_i$ of $G$, in which $v$ is the unique cut vertex belonging to the block $R_i$. An extreme vertex is an \emph{exterior extreme vertex} if it is adjacent to only one cut vertex. Let $S=\{s_1,s_2,...,s_\zeta\}$ be the set of support cut vertices of $G$ and let $\{s_{i1},s_{i2},...,s_{il_i}\}$ be the set of exterior extreme vertices adjacent to $s_i\in S$. Also, let $Q=\{Q_1,Q_2,...,Q_\vartheta\}$ be the set of blocks of $G$ which contain more than one cut vertex and more than one extreme vertex and let $\{q_{i1},q_{i2},...,q_{it_i}\}$ be the set of extreme vertices belonging to $Q_i\in Q$. Now, let $\phi=\displaystyle\max_{1\le i\le \zeta,1\le j \le\vartheta}\{l_i,t_j\}$. With the above notation  we have the following result.

\begin{theorem}
For any generalized tree $G$, $$pd(G)\le \left\{\begin{array}{ll}
                                            \zeta+\vartheta+\phi-1, &\mbox{ if $\,\phi\ge 3$;}  \\
                                            \zeta+\vartheta+1, &\mbox{ if $\,\phi \le 2$.}
                                          \end{array}\right.
$$
\end{theorem}

\begin{proof}
For each support cut vertex $s_i\in S$, let $A_i=\{s_{i1}\}$ and for each block $Q_j\in Q$, let $B_j=\{q_{j1}\}$. Let us suppose $\phi\ge 3$.  For every $j\in \{2,...,l_i\}$ we take $M_{ij}=\{s_{ij}\}$ and, if $l_i<\phi-1$, then for every $j\in \{l_{i+1},...,\phi-1\}$ we consider $M_{ij}=\emptyset$.  Analogously,  for every $j\in \{2,...,t_i\}$ we take $N_{ij}=\{q_{ij}\}$ and,  if $t_i<\phi-1$, then for every $j\in \{t_{i+1},...,\phi-1\}$ we consider $N_{ij}=\emptyset$. Now, let $C_j=\bigcup_{i=1}^{\max \{\zeta,\vartheta\}}(M_{ij}\cup N_{ij})$, with $j\in \{2,...,\phi-1\}$.

Let us prove that $\Pi=\{A,A_1,A_2,...A_{\zeta},B_1,B_2,...,B_{\vartheta},C_2,C_3,...,C_{\phi-1}\}$ is a  resolving partition of $G$, where $A=V-\cup_{i=1}^{\zeta}A_i-\cup_{i=1}^{\vartheta}B_i-\cup_{i=2}^{\phi-1}C_i$.  To begin with, let $x,y$ be two different vertices of $G$. We have the following cases.

Case 1: $x$ is a cut vertex or $y$ is a cut vertex. Let us suppose, for instance, $x$ is a cut vertex. So there exists an extreme vertex $s_{i1}$ such that $x$ belongs to a shortest $y- s_{i1}$ path or $y$ belongs to a shortest $x-s_{i1}$ path. Hence, we have $d(x,A_i)=d(x,s_{i1})\ne d(y,s_{i1})=d(y,A_i)$.

Case 2: $x,y$ are extreme vertices. If $x,y$ belong to the same block of $G$, then $x,y$ belong to different sets of $\Pi$. On the contrary, if $x,y$ belong to different blocks in $G$, then let us suppose there exists an extreme vertex $c$ such that $d(x,c)\le 1$ or $d(y,c)\le 1$. We can suppose $c\in A_i$, for some $i\in \{1,...,\zeta\}$, or $c\in B_j$, for some $j\in \{1,...,\vartheta\}$. Without loss of generality, we suppose that $d(x,c)\le 1$. Since $x$ and $y$ belong to different blocks of $G$, we have $d(y,c)>1$. So we obtain either $d(x,A_i)=d(x,c)\le 1<d(y,c)=d(y,A_i)$ or $d(x,B_j)=d(x,c)\le 1<d(y,c)=d(y,B_j)$.

Now, if there exists no such a vertex $c$, then there exist two blocks $H,K\not\in Q$ with $x\in H$ and $y\in K$, which contain more than one cut vertex and only one extreme vertex. So  $x,y\in A$. Let $u\in H$ be a cut vertex  such that $d(y,u)=\max_{v\in H}d(y,v)$. Hence, there exists an extreme vertex $s_{i1}$ such that $u$ belongs to a shortest $x-s_{i1}$ path and $d(y,s_{i1})=d(y,u)+d(u,s_{i1})$. As $x,y$ belong to different blocks and $d(y,u)=\max_{v\in H}d(y,v)$ we have $d(y,u)\ge 2$. Thus,
\begin{align*}
d(y,A_i)&=d(y,s_{i1})\\
&=d(y,u)+d(u,s_{i1})\\
&\ge 2+d(u,s_{i1})\\
&>1+d(u,s_{i1})\\
&=d(x,u)+d(u,s_{i1})\\
&=d(x,A_i).
\end{align*}
Hence, we conclude that if $\phi\ge 3$, then for every $x,y\in V$, $r(x|\Pi)\ne r(y|\Pi)$. Therefore, $\Pi$ is a resolving partition.

On the other hand, if $\phi\le 2$, then  $\Pi'=\{A,A_1,A_2,...A_{\zeta},B_1,B_2,...,B_{\vartheta}\}$ is a partition of $V$. Proceeding as above we obtain that $\Pi'$ is a resolving partition.
\end{proof}

The above bound is achieved, for instance, for the graph in Figure \ref{cotagentree}, where $\zeta=3$, $\vartheta=0$ and $\phi=3$. Also, notice that for the particular case of trees we have $\zeta=\xi$, $\phi=\theta$ and $\vartheta=0$. So  the above result leads to Corollary \ref{ThPdTrees2}.



\end{document}